\documentclass[12pt]{article}
\usepackage[a4paper,margin=1.0in]{geometry}

\usepackage[colorlinks,citecolor=magenta,linkcolor=black]{hyperref}
\pdfpagewidth=\paperwidth \pdfpageheight=\paperheight
\usepackage{amsfonts,amssymb,amsthm,amsmath,eucal,tabu,url}
\usepackage{pgf}
 \usepackage{array}
 \usepackage{tikz-cd}
 \usepackage{pstricks}
 \usepackage{pstricks-add}
 \usepackage{pgf,tikz}
 \usetikzlibrary{automata}
 \usetikzlibrary{arrows}
 \usepackage{indentfirst}
\usepackage{tabularx} 

\theoremstyle{plain}
\newtheorem{thm}{Theorem}[section]
\newtheorem{theorem}[thm]{Theorem}

\theoremstyle{definition}
\newtheorem{definition}[thm]{Definition}
\newtheorem{remark}[thm]{Remark}
\newtheorem{example}[thm]{Example}

\newtheorem{question}[thm]{Question}

\newtheorem{assume}[thm]{Assumption}

\newtheorem{thevarthm}[thm]{\varthmname}

\newenvironment{varthm*}[1]{\trivlist\item[]{\bf #1.}\it}{\endtrivlist}

\newtheorem{custom}{{\rm Theorem}}

\renewcommand\geq{\geqslant}

\renewcommand\leq{\leqslant}

\newcommand\be{\begin{eqnarray*}}
\newcommand\ee{\end{eqnarray*}}

\newcommand\newop[2]{\def#1{\mathop{\rm #2}\nolimits}}
\newop\log{log}
\newop\ord{ord}
\newop\Gal{Gal}
\newop\SL{SL}
\newop\Bl{Bl}
\newop\mult{mult}
\newop\mass{mass}
\newop\div{div}
\newop\codim{codim}
\newop\sing{sing}
\newop\vdim{vdim}
\newop\edim{edim}
\newop\Ass{Ass}
\newop\size{size}
\newop\reg{reg}
\newop\satdeg{satdeg}
\newop\supp{supp}
\newop\Neg{Neg}
\newop\Nef{Nef}
\newop\Nefh{Nef_H}
\newop\Eff{Eff}
\newop\Zar{Zar}
\newop\MB{MB}
\newop\MBxC{MB\mathit{(x,C)}}
\newop\NnB{NnB}
\newop\Bigg{Big}
\newop\Effbar{\overline{\Eff}}

\def\keywordname{{\bfseries Keywords}}%
\def\keywords#1{\par\addvspace\medskipamount{\rightskip=0pt plus1cm
\def\and{\ifhmode\unskip\nobreak\fi\ $\cdot$
}\noindent\keywordname\enspace\ignorespaces#1\par}}
\def\subclassname{{\bfseries Mathematics Subject Classification
(2020)}\enspace}
\def\subclass#1{\par\addvspace\medskipamount{\rightskip=0pt plus1cm
\def\and{\ifhmode\unskip\nobreak\fi\ $\cdot$
}\noindent\subclassname\ignorespaces#1\par}}

\begin{document}
\title{On free and nearly free arrangements of conics admitting certain ${\rm ADE}$ singularities}
\author{Piotr Pokora}
\date{\today}
\maketitle

\thispagestyle{empty}
\begin{abstract}
The main purpose of this paper is to provide combinatorial constraints on the constructability of free and nearly free arrangements of smooth plane conics admitting certain ${\rm ADE}$ singularites. 
\keywords{conic arrangements, quasi-homogeneous singularities, freeness, nearly freeness}
\subclass{14N20, 14C20, 32S22}
\end{abstract}
\section{Introduction}
The theory of lines arrangements in the plane is an almost ancient topic in geometry and combinatorics. Recently, arrangements of rational plane curves have received a lot of attention, mostly in the context of studies devoted to the freeness of curves, see for instance \cite{DimcaConic,IDPS,DimcaPokora,max,SchenckToh,Schen}. Classically, the freeness problem is related with (central) hyperplane arrangements, and this is mostly due to the celebrated conjecture by Terao. However, it turns out very shortly that the freeness problem of plane curve is getting more and more difficult if we increase degrees of irreducible components, it is especially visible if we pass from line arrangements to conic-line arrangements, or just conic arrangements in the plane. The main aim of the present note is to understand basic algebraic and combinatorial properties of free and nearly free plane curves in the complex projective plane that are constructed using smooth plane conics as irreducible components. Let us recall the following basic definition and then we formulate our main results devoted to arrangements of smooth plane conics.

Let us denote by $S := \mathbb{C}[x,y,z]$ the coordinate ring of $\mathbb{P}^{2}_{\mathbb{C}}$ and for a homogeneous polynomial $f \in S$ let us denote by $J_{f}$ the Jacobian ideal associated with $f$, i.e., the ideal generated by the partial derivatives $\partial_{x}\, f, \partial_{y} \, f, \partial_{z} \, f$.

Let $C : f=0$ be a reduced curve in $\mathbb{P}^{2}_{\mathbb{C}}$ of degree $d$ defined by $f \in S$. Denote by $M(f) := S/ J_{f}$ the Milnor algebra. 
\begin{definition}
We say that $C$ is $m$-syzygy when $M(f)$ has the following minimal graded free resolution:
$$0 \rightarrow \bigoplus_{i=1}^{m-2}S(-e_{i}) \rightarrow \bigoplus_{i=1}^{m}S(1-d - d_{i}) \rightarrow S^{3}(1-d)\rightarrow S$$
with $e_{1} \leq e_{2} \leq ... \leq e_{m-2}$ and $1\leq d_{1} \leq ... \leq d_{m}$.
\end{definition}
In the light of the above definition, we define the minimal degree of the Jacobian relations among the partial derivatives of $f$, namely
$${\rm mdr}(f) := d_{1}.$$

Among many examples of $m$-syzygy plane curves, we can distinguish the following important classes, and this can be done via the above homological description for a reduced plane curve $C : f=0$ and by using the information about the total Tjurina number of $C$. Before we do so, let us recall the following crucial definition that we will use throughout this note.
\begin{definition}

Let $p$ be an isolated singularity of a polynomial $f\in \mathbb{C}[x,y]$. Since we can change the local coordinates, let $p=(0,0)$.
The number 
$$\mu_{p}=\dim_\mathbb{C}\left(\mathbb{C}[x,y] /\bigg\langle \frac{\partial f}{\partial x},\frac{\partial f}{\partial y} \bigg\rangle\right)$$
is called the Milnor number of $f$ at $p$.

The number
$$\tau_{p}=\dim_\mathbb{C}\left(\mathbb{C}[ x,y] /\bigg\langle f,\frac{\partial f}{\partial x},\frac{\partial f}{\partial y}\bigg\rangle \right)$$
is called the Tjurina number of $f$ at $p$.
\end{definition}
For a projective situation, with a point $p\in \mathbb{P}^{2}_{\mathbb{C}}$ and a homogeneous polynomial $f\in \mathbb{C}[x,y,z]$, we take local affine coordinates such that $p=(0,0,1)$ and then the dehomogenization of $f$.

Finally, the total Tjurina number of a given reduced curve $C \subset \mathbb{P}^{2}_{\mathbb{C}}$ is defined as
$$\tau(C) = \sum_{p \in {\rm Sing}(C)} \tau_{p}.$$
\begin{definition}
We say that 
\begin{itemize}
\item $C$ is free if and only if $m=2$ and $d_{1}+d_{2}=d-1$. Moreover, \cite{duP} tells us that $C$ with ${\rm mdr}(f)\leq (d-1)/2$ is free if and only if
\begin{equation}
\label{duPles}
(d-1)^{2} - d_{1}(d-d_{1}-1) = \tau(C).
\end{equation}
\item $C$ is nearly-free if and only if $m=3$, $d_{1}+d_{2} = d$, $d_{2}=d_{3}$, and $e_{1}=d+d_{2}$. Moreover, by a result due to Dimca \cite{Dimca1}, we know that $C$ is nearly free if and only if
\begin{equation}
(d-1)^{2}-d_{1}(d-d_{1}-1)=\tau(C)+1.
\end{equation}
\end{itemize}
\end{definition}
In the context of plane conic arrangements, we have two results that stand for the main motivation for this paper. First of all, Dimca, Janasz and the author in \cite{DJP} observed that conic arrangements admitting only nodes and tacnodes as singularities are never free, and they provided a complete classification of such arrangements that are nearly free
\begin{theorem}[Janasz-Dimca-Pokora]
Let $\mathcal{C} \subset \mathbb{P}^{2}_{\mathbb{C}}$ be an arrangement of $k \geq 2$ smooth conics with only nodes and tacnodes as singularities. Then $\mathcal{C}$ is nearly-free if and only if
$$ k \leq 4 \text{ and the number of tacnodes is equal to } k(k-1).$$
\end{theorem}
An analogous question for conic arrangements, where the list of possible singularities was extended, was recently considered by the author in \cite[Theorem B]{Pokora}. Among others, the author proved the following characterization.
\begin{theorem}[Pokora]
\label{PokC}
Let $\mathcal{C} \subset \mathbb{P}^{2}_{\mathbb{C}}$ be an arrangement of $k \geq 2$ smooth conics with only nodes, tacnodes, ordinary triple and quadruple points as singularities. Then $\mathcal{C}$ is never free.
\end{theorem}

Our main aim is to continue this path of studies and we want to understand to what extent one can generalize the aforementioned results. In the paper we focus on the following class of smooth conic arrangements in the plane.
\begin{assume}
\label{quass}
Let $\mathcal{C} = \{C_{1}, ..., C_{k}\} \subset \mathbb{P}^{2}_{\mathbb{C}}$ be an arrangement of $k\geq 2$ smooth conics. From now on we restrict our attention to arrangements of conics that have $n_{2}$ nodes, $n_{3}$ ordinary triple points, $t_{3}$ tacnodes, $t_{5}$ points of type $A_{5}$, and $t_{7}$ points of type $A_{7}$.
\end{assume}
The above assumption means that these are the worst singularities for conic arrangements that we are allowed to use here. 

Due to the notation reasons, and for the completeness of the note, let us recall the classification of ${\rm ADE}$ singularities for plane curves:
\begin{center}
\begin{tabular}{ll}
$A_{k}$ with $k\geq 1$ & $: \, x^{2}+y^{k+1}  = 0$, \\
$D_{k}$ with $k\geq 4$ & $: \, y^{2}x + x^{k-1}  = 0$,\\
$E_{6}$ & $: \, x^{3} + y^{4} = 0$, \\
$E_{7}$ & $: \, x^{3} + xy^{3} = 0$, \\
$E_{8}$ & $:\, x^{3}+y^{5} = 0$.
\end{tabular}
\end{center}
Our selection of singularities is well-motivated by works that have been published recently, see \cite{DimcaSt} as our main inspiration. Of course, all prescribed singularities in Assumption \ref{quass} are \textit{quasi-homogeneous} since these are ${\rm ADE}$ singularities.
\begin{definition}
A singularity is called quasi-homogeneous if and only if there exists a holomorphic change of variables so that the defining equation becomes weighted homogeneous.
\end{definition}
 Recall that $f(x,y) = \sum_{i,j}c_{i,j}x^{i}y^{j}$ is weighted homogeneous if there exist rational numbers $\alpha, \beta$ such that $\sum_{i,j} c_{i,j}x^{i\cdot \alpha} y^{j \cdot \beta}$ is homogeneous. One can show that if $f(x,y)$ is a convergent power series with an isolated singularity at the origin, then $f(x,y)$ is in the ideal generated by the partial derivatives if and only if $f$ is quasi-homogeneous. It means that in the quasi-homogeneous case one has $\tau_{p} = \mu_{p}$, i.e., the local Tjurina number of $p$ is equal to the local Milnor number of $p$. This technical observation will be used freely in the paper. 

 Let us now present the main results of the paper. Our first result is devoted to conic arrangements which admit only nodes and ordinary triple points as singularities.
 \begin{custom}[see Theorem \ref{near}]
Let $\mathcal{C} = \{C_{1}, ..., C_{k}\} \subset \mathbb{P}^{2}_{\mathbb{C}}$ be an arrangement of $k\geq 2$ smooth conics that admits $n_{2}$ nodes and $n_{3}$ ordinary triple points. Then $\mathcal{C}$ is never nearly free.
 \end{custom}
Notice that this is in some ways at odds with what we know about line arrangements with nodes and triple points since there are examples of line arrangements with double and triple points that are nearly free \cite{Kabat}. 

Then we study free arrangements of conics with prescribed singularities according to Assumption \ref{quass}. We can show the following strong bound on the degree of such conic arrangements.
\begin{custom}[see Theorem \ref{char}]
Let $\mathcal{C} \subset \mathbb{P}^{2}_{\mathbb{C}}$ be an arrangement of $k\geq 2$ conics satisfying Assumption \ref{quass}. Assume that $\mathcal{C}$ is free, then $ k \in \{2,3,4\}$.
\end{custom}
Then we focus on constructions of new examples of conic arrangements with prescribed singularities that are nearly free. 
\begin{custom}[see Theorem \ref{deform}]
Let $C \, : f=0$ be a reduced plane curve in $\mathbb{P}^{2}_{\mathbb{C}}$ that admits only ${\rm ADE}$ singularities such that it has at least one tacnode. Assume that there exists a deformation $C' \, : f^{'}=0$ which is obtained from $C$ by the following procedure:
\begin{itemize}
    \item a tacnode is deformed into two nodes;
    \item all other singular points maintain their types.
\end{itemize}
Assume furthermore that $\eta(C) = \eta(C')$, then $C'$ is nearly free.
\end{custom}
As an example, we present a direct application of the above result using Persson's triconical arrangement as our starting point.

In the last part of the paper we study examples of pencils of conics in the plane, but here we allow them to possess singularities that are arbitrary (i.e., not quasi-homogeneous and/or of arbitrary multiplicity). We present three particular families of conic pencils in the plane, one is free, the second one is neither free nor nearly free, the third one is nearly free, but all of them possess an interesting combinatorial property that we call the combinatorial supersolvability. Our main observation is that the property of being combinatorially supersolvable in the class of conic arrangements is not strong enough to produce a new closed class of free conic arrangements.
\section{On the freeness and nearly freeness of conic arrangements in the plane}
\begin{theorem}
\label{near}
   Let $\mathcal{C} = \{C_{1}, ..., C_{k}\} \subset \mathbb{P}^{2}_{\mathbb{C}}$ be an arrangement of $k\geq 2$ smooth conics that admits $n_{2}$ nodes and $n_{3}$ ordinary triple points. Then $\mathcal{C}$ is never nearly free. 
\end{theorem}
\begin{proof}
Our proof is very direct and we proceed by contradiction.  Assume that $\mathcal{C} \, : f=0$ is nearly free. Then it satisfies the following equation
$$d_{1}^{2} - d_{1}(2k-1)+(2k-1)^{2} = \tau(\mathcal{C}) + 1 = n_{2} + 4n_{3}+1,$$
where the formula for the total Tjurina number comes from the fact that all singularities are quasi-homogeneous and $d_{1} = {\rm mdr}(f)$.
By the combinatorial count for conic arrangements with nodes and ordinary triple points we know that 
$$2(k^{2}-k) = n_{2} + 3n_{3}.$$
Then we obtain the following
$$d_{1}^{2} - d_{1}(2k-1) + 4k^{2}-4k+1 = n_{3} + 2k^{2}-2k+1$$
and we finally obtain
$$d_{1}^{2}-d_{1}(2k-1) + 2k^{2}-2k-n_{3}=0.$$
Since $\mathcal{C}$ is nearly free, then
$$\triangle_{d_{1}} = (2k-1)^{2}-4(2k^{2}-2k-n_{3}) \geq 0.$$
This gives us that
$$n_{3}\geq k^{2}-k-\frac{1}{4}.$$
Coming back to the combinatorial count, we have
$$2k^{2}-2k = n_{2} + 3n_{3} \geq n_{2} + n_{3} + 2\bigg(k^{2}-k-\frac{1}{4}\bigg),$$
so we finally obtain $n_{2}+n_{3}\leq \frac{1}{2}$. Since $n_{2},n_{3}\geq 0$ and $n_{2}+n_{3}>1$, we arrive at the contradiction.
\end{proof}
Based on results presented up to right now in the paper, we see a bright contrast with the world of line arrangements (and conic-line arrangements) compared with what is happening in the class of conic arrangements in the plane. Our first observation shows that it is very difficult to construct examples of conic arrangements that are free or nearly free. In the known cases of such arrangements we exploit singularities of type $A_{3}$, $A_{5}$ and $A_{7}$ which is very interesting and yet not well understood.
In that context, we show the following result strong bound on the number of irreducible components for conic arrangements satisfying Assumption \ref{quass} -- it turns out that such arrangements are extremely rare.
\begin{theorem}
\label{char}
Let $\mathcal{C} \subset \mathbb{P}^{2}_{\mathbb{C}}$ be an arrangement of $k\geq 2$ conics satisfying Assumption \ref{quass}. Assume that $\mathcal{C}$ is free, then $ k \in \{2,3,4\}$.
\end{theorem}
Before we pass to the proof of the above theorem, let us recall the following result \cite[Theorem 2.1]{DimcaSernesi}.
\begin{theorem}[Dimca-Sernesi]
\label{sern}
Let $C \, : \, f = 0$ be a reduced curve of degree $d$ in $\mathbb{P}^{2}_{\mathbb{C}}$ having only quasi-homogeneous singularities. Then $${\rm mdr}(f) \geq \alpha_{C}\cdot d - 2,$$
where $\alpha_{C}$ denotes the Arnold exponent of $C$.
\end{theorem}
It is worth recalling that the Arnold exponent of a given reduced curve $C \subset \mathbb{P}^{2}_{\mathbb{C}}$ is defined as the minimum over all Arnold exponents of singular points $p$ in $C$. In modern language, the Arnold exponents of singular points are nothing else but the log canonical thresholds of singularities. Let us explain how to compute the log canonical threshold in our setting. Since all the singularities we consider here are quasi-homogeneous, we can use the following pattern (cf. \cite[Formula 2.1]{DimcaSernesi}). 

Recall that the germ $(C,p)$ is weighted homogeneous of type $(w_{1},w_{2};1)$ with $0 < w_{j} \leq 1/2$ if there are local analytic coordinates $y_{1},y_{2}$ centered at $p=(0,0)$ and a polynomial $g(y_{1},y_{2})= \sum_{u,v} c_{u, v} y_{1}^{u} y_{2}^{v}$ with $c_{u,v} \in \mathbb{C}$, where the sum is over all pairs $(u,v) \in \mathbb{N}^{2}$ with $u w_{1} + v w_{2}=1$. Using this description, the log canonical threshold can be defined as $${\rm lct}_{p}(g) := w_{1}+w_{2}.$$
\begin{example}
Consider $A_{7}$ singularity at $p=(0,0)$ with the normal form $g(x,y) = y^{2} + x^{8}$. Then $w_{1}=\frac{1}{2}, w_{2} = \frac{1}{8}$, and hence we have ${\rm lct}_{p}(g) = \frac{5}{8}$.
\end{example}

Now we can present our proof of Theorem \ref{char}.
\begin{proof}
Let $\mathcal{C}$ be an arrangement of $k$ conics satisfying Assumption \ref{quass}. Then the Arnold exponent of $\mathcal{C}$ is equal to $\alpha_{\mathcal{C}} = \frac{5}{8}$ since the minimum is obtained for a singular point of type $A_{7}$.
Using Theorem \ref{sern}, we have
$$\frac{2k-1}{2} \geq d_{1} \geq \frac{5}{8}\cdot 2k-2.$$
It implies that
$$4k-2 \geq 5k - 8,$$
and we arrive at
$$k \leq 6.$$
Now we exclude cases where $k=6$ and $k=5$.

Assume that $k=6$. It means, by Theorem \ref{sern}, that we have the following lower bound on $d_{1}$:
$$d_{1} \geq \frac{5}{8}\cdot 12 - 2 = 5.5,$$
and this means that $d_{1} \geq 6$. However, $d_{1}$ has to be less than or equal to $5$, a contradiction.

Assume now that $k=5$. It means again, by Theorem \ref{sern}, that we have the following lower bound on $d_{1}$:
$$d_{1} \geq \frac{5}{8}\cdot 10 - 2 = 4.25,$$
and this means that $d_{1} \geq 5$. However, $d_{1}$ has to be less than or equal to $4$, a contradiction. 
\end{proof}
To complete our degree-wise classification result, we need to construct arrangements with $k \in \{2,3,4\}$ conics and prescribed above singularities that are free. 
\begin{remark}
Let us start with the case $k=2$. We should have an arrangement of two conics such that the total Tjurina number is equal to $7$. Based on a precise discussion in \cite[Proposition 5.5]{DimcaSt}, we know that there exists a $1$-parameter family of two conics that intersect at the single point being $A_{7}$ singularity. Furthermore, we can check using \verb}SINGULAR} that $d_{1}=1$, and using \eqref{duPles} we conclude that our arrangement is free. 

Now we move on to the case $k=3$, and our example here comes from \cite[Remark 4.2.3]{Celal}. Consider the following arrangement of conics $\mathcal{C} = \{C_{1}, C_{2}, C_{3}\} \subset \mathbb{P}^{2}_{\mathbb{C}}$ with
\begin{equation*}
\begin{array}{l}
C_{1} : -3x^2 + xy + yz + zx = 0, \\
C_{2} : -3y^2 + xy + yz + zx = 0, \\
C_{3} : -3z^2 + xy + yz + zx = 0.
\end{array}
\end{equation*}
This arrangement has exactly $3$ singularities of type $A_{5}$ and one ordinary triple point, so the total Tjurina number is equal to $19$. Using \verb}SINGULAR}, one can check that $d_{1}=2$, and this gives us
$$d_{1}^2 - d_{1}(2k-1) + (2k-1)^2 = 19 = \tau(\mathcal{C}) = 4n_{3} + 5t_{5} =19,$$
hence $\mathcal{C}$ is free.

It is natural to ask what to say about the case $k=4$. Using the same argument as in our proof of Theorem \ref{char}, we obtain that $d_{1}=3$, so we arrive at the situation that $\mathcal{C}$ has to be a maximizing octic curve, see \cite[Theorem 2.9]{max} for all necessary details. However, this is a very subtle and complicated problem, since there are exactly $50$ admissible weak combinatorial types of conic arrangements that can potentially lead to a free arrangement. We are aware of a free conic arrangement with $k=4$ conics with only ${\rm ADE}$ singularities, but in that case the singular locus consists of two points of type $D_{10}$, two points of type $D_{6}$, one tacnode and two nodes, see \cite[Example 3.3]{max}. Moreover, we have the following arrangement of conics, inspired by Persson's triconical arrangement \cite[2.1 Proposition]{Persson}, that we would like to describe below.
\end{remark}
\begin{example}

\label{4pers}
This construction delivers an arrangement of $4$ conics with exactly $4$ singularities of type $A_{7}$.
Let us consider the arrangement $\mathcal{P}_{4}$ given by the following homogeneous equation:
$$Q(x,y,z) = (x^{2}+y^{2}-z^{2})\cdot(2x^{2}+y^{2}+2xz)\cdot(x^{2}+y^{2}+2xz)\cdot(4x^{2}+6y^{2}+4xz - 9z^2).$$
The four $A_{7}$ singularites are located at
$$P_{1}=(-1:0:1), \quad P_{2}=(0:0:1), \quad P_{3}=(-2:0:1), \quad P_{4}=(1:0:1),$$
so all these $A_{7}$ singularities are collinear. We have also $8$ nodes as additional singularities, so altogether we have an arrangement of $4$ smooth conics with $4$ singularities of type $A_{7}$ and $8$ nodes. We can check rather easily that the arrangement is nearly free. In order to do so, we only need to check, via \verb}SINGULAR}, that $d_{1} = 3$, and then 
$$37=d_{1}^{2}-d_{1}(2k-1)+(2k-1)^2 = 9-21 + 49=\tau(\mathcal{P}_{4})+1=4\cdot 7 + 8\cdot 1 + 1.$$
\end{example}
As we can see in the example above, we are very close to a free example of plane conics, but unfortunately we cannot do any better, so the case $k=4$ remains to be decided.
\begin{remark}
Using almost the same argument as in Theorem \ref{char}, we can show that if $\mathcal{C}$ is a nearly free arrangement of $k\geq 2$ conics satisfying Assumption \ref{quass}, then
$$k \leq 8.$$
We know that there are examples of nearly free conic arrangements with only tacnodes as singularities, and in this setting $k \in \{2,3,4\}$. Moreover, by Example \ref{4pers}, we have an additional arrangement of $4$ conics that is nearly free. On the other hand, we are not aware of a single example of a nearly free conic arrangement with $k \in \{5,6,7,8\}$ and singularities prescribed in Assumption \ref{quass}.
\end{remark}
\section{Constructing nearly free arrangements from free ones}
Let us introduce the following definition.
\begin{definition}
Let $C \, : f=0$ be a reduced plane curve of degree $d$ in $\mathbb{P}^{2}_{\mathbb{C}}$ and let ${\rm mdr}(f) = d_{1}$. Then we define
$$\eta(C) = d_{1}^{2} -d_{1}(d-1) + (d-1)^{2}.$$
\end{definition}
Our main result is the following.
\begin{theorem}
\label{deform}
Let $C \, : f=0$ be a reduced plane curve in $\mathbb{P}^{2}_{\mathbb{C}}$ that admits only ${\rm ADE}$ singularities such that it has at least one tacnode. Assume that there exists a deformation $C' \, : f^{'}=0$ which is obtained from $C$ by the following procedure:
\begin{itemize}
    \item a tacnode is deformed into two nodes;
    \item all other singular points maintain their types.
\end{itemize}
Assume furthermore that $\eta(C) = \eta(C')$, then $C'$ is nearly free.
\end{theorem}
\begin{proof}
Assume that $C$ is free, then the following equation is satisfied:
$$\eta(C) = d_{1}^{2}-d_{1}(d-1)+(d-1)^{2} = \tau(C).$$
Now we look at $\tau(C')$. By the above assumptions observe that
$$\tau(C') = (n_{2}+2) + 3(t_{3}-1) + \ldots = -1 + \tau(C),$$
so we arrive at the following equation
$$\tau(C) = \tau(C')+1.$$
Since $\eta(C) = \eta(C')$, we have
$$\eta(C') = \tau(C')+1,$$
which means that $C'$ is nearly free.
\end{proof}
Now we are going to use the above result in practice.
\begin{example}
\label{pers}
Consider $\mathcal{C} = \{C_{1}, C_{2}, C_{3}\} \subset \mathbb{P}^{2}_{\mathbb{C}}$ with the defining equation
$$F(x,y,z) = (x^{2}+y^{2}-z^{2})(2x^2 + y^2 +2xz)(2x^2 + y^2 - 2xz) =0.$$
This arrangement is the aforementioned Persson's triconical arrangement \cite[2.1 Proposition]{Persson}. We have the following list of singular points:
$$n_{2} = 2, \quad t_{3}=1, \quad t_{7}=2,$$
and the arrangement $\mathcal{C}$ is free since this is an example of a maximizing sextic in the plane \cite{max}.

Consider now the following deformation $\mathcal{C}^{'}$ of $\mathcal{C}$ given by the polynomial
$$G(x,y,z) = (2x^{2}+2y^{2}+3xz+z^{2})(2x^{2}+2y^{2}-3xz+z^{2})(x^{2}+4y^{2}-z^{2}).$$
It can be easily checked that this deformation satisfies the assumptions of Theorem \ref{deform} -- we have
$n_{2}=4, \,t_{7}=2$, $d_{1}={\rm mdr}(G)=3$, and $\eta(F) = \eta(G)$, which leads us to the fact that the arrangement $\mathcal{C}'$ is nearly free.
\end{example}
\section{On the combinatorial supersolvability of conic arrangements}
Our main aim here is test certain combinatorial conditions in order to construct new free conic arrangements with arbitrary singularities. Our decision to allow different types of singularities is based on our current experience that there are only a few known free conic arrangements with ${\rm ADE}$ singularities. In order to approach this problem we introduce now a certain combinatorial property that might potentially help with constructing new examples of free arrangements.
\begin{definition}
Let $C \subset \mathbb{P}^{2}_{\mathbb{C}}$ be a reduced plane curve which has at least two irreducible components. We say that $C$ is \textbf{combinatorially supersolvable} if there exists a singular point $p \in {\rm Sing}(C)$ such that for any singular point $q \in {\rm Sing}(C)$ there exists an irreducible component $C'$ of $C$ such that $p,q \in C'$. 
\end{definition}
 In the setting of the above definition, the singular point $p$ is called \textbf{combinatorially modular} and, in general, such a point is not uniquely determined. The above definition is a natural geometric generalization of the notion of being supersolvable in the setting of line arrangements. Let us recall here that if $\mathcal{L}\subset \mathbb{P}^{2}_{\mathbb{C}}$ is an arrangement of lines that is supersolvable, then $\mathcal{L}$ is free, and this result follows from \cite{Jambu}. It is natural to wonder whether being combinatorially supersolvable in the class of conic arrangements (or just reduced plane curves) is related with the freeness. Let us present some contrasting examples.
 \begin{example}[P\l oski's moustache curve]
    Here we consider a curve $\mathcal{P}_{2m}$ of degree $2m$ which consists of exactly $m$ smooth conics. We can take the following homogeneous equation for all necessary considerations:
$$Q(x,y,z) = \prod_{i=1}^{m}(xz+i\cdot x^2 + y^{2}).$$
By a result due to P\l oski \cite{Ploski} we know that the above arrangement of curves has the only one singular point with the local Milnor number equal to $(2m-1)^{2}-m$. On the other hand, its local Tjurina number is equal to $(2m-1)^{2}-(2m-2)$, so this singularity is not quasi-homogeneous. A deeper inspection tells us that the considered singularity is a merger of $A_{7}$ singularities produced by each pair of distinct conics. One can check (almost by hand) that $d_{1} = {\rm mdr}(Q)=1$, and hence
$$d_{1}^{2} - d_{1}(d-1) + (d-1)^{2} - \tau(\mathcal{P}_{2m}) = 1 - (2m-1) + (2m-1)^{2} - ((2m-1)^{2}-2m+2)=0,$$
so for every $m\geq 2$ the curve is free. Obviously $\mathcal{P}_{2m}$ is combinatorially supersolvable.
 \end{example}
 \begin{example}[A pencil of conics passing through $4$ general points]
 Consider the following two conics given by
 $$F(x,y,z) = 3x^2+y^2-4z^2, \quad G(x,y,z) = x^2 + 3y^2-4z^2.$$
 It is easy to see that these two conics intersect at $4$ distinct general points. Now we take the following pencil $\mathcal{PC}$ given by the following polynomial:
 $$H(x,y,z) = f g \cdot \prod_{i=2}^{m-2}(f+i\cdot g).$$
 It is well-known that all four $k$-fold intersection points are quasi-homogeneous -- see for instance \cite[Proposition 4.3]{DJP}. Now we are going to check whether $\mathcal{PC}$ is free. Observe that ${\rm mdr}(H)=2$ since we have the following relation (that can be easily check via \verb}Singular} \cite{Singular}):
 $$yz \cdot \partial_{x} H + xz \cdot \partial_{y} H + xy\cdot \partial_{z} H=0.$$
Now we can compute the so-called \textbf{defect}. Let us recall that the defect of a given reduced curve $C \subset \mathbb{P}^{2}_{\mathbb{C}}$ of degree $d$ is defined as
$$\nu(C) = d_{1}^{2}-d_{1}(d-1) + (d-1)^{2} - \tau(C).$$
Having this formula in hand, we can compute that
 $$\nu(\mathcal{PC}) = 4 - 2(2m-1) + (2m-1)^{2} - 4\cdot(m-1)^{2} = 3,$$
 so $\mathcal{PC}$ is neither free nor nearly-free. Obviously $\mathcal{PC}$ is combinatorially supersolvable.
 \end{example}
\begin{example}[Pencil of conics with $2$ base points]
This construction was presented \cite[Section 3]{Malara} and was inspired by \cite{DimcaConic}. Let us consider the following curve $\mathcal{C}_{k}$ of degree $2k$ consisting of exactly $k$ conics which is given by the following homogeneous equation:
$$F(x,y,z) = x^{k}y^{k}+z^{2k}$$
with $k\geq 2$. It is easy to observe that $\mathcal{C}_{m}$ has only two singular points located at $P_{1}=(1:0:0)$ and $P_{2}=(0:1:0)$, so we have a pencil of conics with two base points. Now we compute the defect $\nu(\mathcal{C}_{k})$ in order to show that our arrangement of conics is nearly free. Observe that $d_{1}={\rm mdr}(F)=1$ and it follows from the obvious check that
$$x\cdot \partial_{x}F - y\cdot \partial_{y}F = 0.$$ Moreover, the singular points are quasi-homogeneous and their local Tjurina numbers are equal to $(2k-1)(k-1)$. Finally we are in a position to compute the defect, namely
$$\nu(\mathcal{C}_{k}) = 1 - (2k-1)+(2k-1)^{2} -2\cdot(2k-1)(k-1) = 1,$$
and this justifies our claim that the arrangements is nearly free. Obviously $\mathcal{C}_{k}$ is combinatorially supersolvable for each $k \geq 2$.
\end{example}
 These families of examples show that a naive combinatorial generalization of the supersolvability, taken almost literally from the word of line arrangements in the plane, is not sufficient to produce a new closed class of free conic arrangements. However, in certain situations we get new interesting examples of free or nearly free plane curves which is a kind of a consolation prize.

Finishing this section, let us recall the following crucial definition which apparently might have higher potential for further research - this notion was introduced recently in \cite{IDPS}.
\begin{definition}
Let $C \subset \mathbb{P}^{2}_{\mathbb{C}}$ be a reduced curve. We say that $p 
\in C$ is a modular point for $C$ if the central projection
$$\pi_{p} : \mathbb{P}^{2}_{\mathbb{C}} \setminus \{p\} \rightarrow \mathbb{P}^{1}_{\mathbb{C}}$$
induces a locally trivial fibration of the complement $M(C) = \mathbb{P}^{2}_{\mathbb{C}} \setminus C$. We say that a given curve $C$ is \textbf{supersolvable} if it has at least one modular point.
\end{definition}
 If we restrict to the class of line arrangements, the above definition of a modular point coincides with the usual one. In this setting, the following remains open.
 \begin{question}
 Is a supersolvable plane curve $C$ always free?
\end{question}
We do not dare to decide whether this conjecture is true or false here, but it would be very natural to compare the supersolvability and the combinatorial supersolvability using different classes of curves, i.e., not only conics in the plane.

\section{Addendum to \cite{Pokora}}
Here we would like point out an observation regarding non-freeness of the so-called $d$-arrangements of plane curves in $\mathbb{P}^{2}_{\mathbb{C}}$.
\begin{definition}
Let $\mathcal{D}=\{C_{1}, ..., C_{k}\} \subset \mathbb{P}^{2}_{\mathbb{C}}$ be an arrangement of irreducible curves. We say that $\mathcal{D}$ is a $d$-arrangement if the following conditions hold:
\begin{enumerate}
    \item all irreducible components of $\mathcal{D}$ are smooth of degree $d\geq 1$,
    \item all intersection points are ordinary singularities.  
\end{enumerate}
\end{definition}
In the class of $d$-arrangements, the quasi-homogeneity property for singularities is, generally speaking, governed by their multiplicity, namely ordinary singularities are quasi-homogeneous if their multiplicity is less than $5$ -- see \cite[Exercise 7.31]{RCS}.
Based on that observation we can formulate the following result that generalizes \cite[Theorem A]{Pokora}.
\begin{theorem}
Let $\mathcal{D}\subset \mathbb{P}^{2}_{\mathbb{C}}$ be a $d$-arrangement with $d\geq 3$ and $k\geq 3$ having ordinary singularities with multiplicities less than $5$. Then $\mathcal{D}$ is never free.
\end{theorem}
\begin{proof}
Since the proof is verbatim as in \cite[Theorem A]{Pokora}, with suitable small changes with respect to the total Tjurina number of a given arrangement $\mathcal{D}$ and the combinatorial count
$$d\cdot \binom{k}{2} = n_{2} + 3n_{3} + 6 n_{4},$$
where $n_{i}$ denotes the number of (ordinary) points of multiplicity $i$, we leave the details to the reader.
\end{proof}
\section*{Conflict of Interests}
I declare that there is no conflict of interest regarding the publication of this paper.
\section*{Acknowledgments}
I would like to thank Alex Dimca for many interesting conversations around the content of the paper and remarks. In addition, I am very grateful to an anonymous referee for all the valuable comments that have allowed me to improve the paper.

Piotr Pokora is partially supported by The Excellent Small Working Groups Programme \textbf{DNWZ.711/IDUB/ESWG/2023/01/00002} at the University of the National Education Commission Krakow.

\vskip 0.5 cm

\bigskip
Piotr Pokora,
Department of Mathematics,
University of the National Education Commission Krakow,
Podchor\c a\.zych 2,
PL-30-084 Krak\'ow, Poland. \\
\nopagebreak
\textit{E-mail address:} \texttt{piotr.pokora@up.krakow.pl}

\end{document}